\documentclass{amsart}
\usepackage{amssymb}
\usepackage{bbold}
\copyrightinfo{2015}{V.X. Genest \emph{et al.}}

\newtheorem{prop}{Proposition}
\theoremstyle{definition}

\theoremstyle{remark}

\numberwithin{equation}{section}
\newcommand{\mb}{\mathbf}
\newcommand{\mf}{\mathfrak}

\begin{document}
\title[]{The non-symmetric Wilson polynomials are the Bannai--Ito polynomials}
\author[V.X. Genest]{Vincent X. genest}
\address{Centre de recherches math\'ematiques, Universit\'e de Montr\'eal, Montr\'eal, QC, Canada H3C 3J7}
\email{genestvi@crm.umontreal.ca}
\author[L. Vinet]{Luc Vinet}
\address{Centre de recherches math\'ematiques, Universit\'e de Montr\'eal, Montr\'eal, QC, Canada H3C 3J7}
\email{vinetl@crm.umontreal.ca}
\author[A. Zhedanov]{Alexei Zhedanov}
\address{Donetsk Institute for Physics and Technology, Donetsk 83114, Ukraine}
\email{zhedanov@yahoo.com}
\date{}
\dedicatory{}
\begin{abstract}
The one-variable non-symmetric Wilson polynomials are shown to coincide with the Bannai--Ito polynomials. The isomorphism between the corresponding degenerate double affine Hecke algebra of type $(C_1^{\vee}, C_1)$ and the Bannai--Ito algebra is established. The Bannai--Ito polynomials are seen to satisfy an ortho\-gonality relation with respect to a positive-definite and continuous measure on the real line. A non-compact form of the Bannai--Ito algebra is introduced and a four-parameter family of its infinite-dimensional and self-adjoint representations is exhibited.
\end{abstract}
\maketitle
\section{Introduction}
The purpose of this paper is fourfold: first, to show that the non-sym\-metric Wilson polynomials coincide with the Bannai--Ito polynomials; second, to identify the corresponding degenerate double affine Hecke algebra of type $(C_1^{\vee}, C_1)$ with the Bannai--Ito algebra; third, to exhibit a positive-definite, continuous ortho\-gonality measure on the real line for the Bannai--Ito polynomials; and also, to present a non-compact form of the Bannai--Ito algebra and display a four-parameter family of its infinite-dimensional, self-adjoint representations. 

We begin by reviewing the essentials of these two families of orthogonal functions and their associated algebraic structures.
\subsection{The Bannai--Ito polynomials and algebra}
The Bannai--Ito (BI) polynomials first appeared in the context of $P-$ and $Q-$ polynomial association schemes. In their book \cite{1984_Bannai&Ito}, E. Bannai and T. Ito identified the polynomials bearing their names as $q=-1$ analogs of the $q$-Racah polynomials in their classification of the orthogonal polynomials satisfying the Leonard duality \cite{1982_Leonard_SIAMJMathAnal_13_656}; see \cite{1990_Bannai} for an overview. 

The (monic)  Bannai--Ito polynomials $\mb{B}_{n}(x)$ of degree $n$ in $x$, which depend on four parameters $\mf{a}, \mf{b}, \mf{c},\mf{d}$ are defined by the three-term recurrence relation\footnote{\noindent With respect to the usual parametrization (e.g \cite{2012_Tsujimoto&Vinet&Zhedanov_AdvMath_229_2123}) where the BI polynomials are denoted by $B_n(y;\rho_1,\rho_2,r_1,r_2)$, we have taken $y=x/2-1/4$, $\rho_1=\mf{a}+1/4$, $\rho_2=\mf{b}+1/4$, $r_1=-\mf{c}-1/4$, $r_2=-\mf{d}-1/4$.}
\begin{align}
\label{BI-3Term}
x\mb{B}_{n}(x)=\mb{B}_{n+1}(x)+(2\mf{a}+1-A_{n}-C_{n})\mb{B}_{n}(x)+A_{n-1}C_{n}\mb{B}_{n-1}(x),
\end{align}
with $\mb{B}_{-1}(x)=0$, $\mb{B}_0(x)=1$ and where the coefficients read
\begin{align*}
A_{n}&=
\begin{cases}
\frac{(n+2\mf{a}+2\mf{c}+2)(n+2\mf{a}+2\mf{d}+2)}{2(n+\mf{a}+\mf{b}+\mf{c}+\mf{d}+2)} & \text{$n$ even}
\\
\frac{(n+2\mf{a}+2\mf{b}+2)(n+2\mf{a}+2\mf{b}+2\mf{c}+2\mf{d}+3)}{2(n+\mf{a}+\mf{b}+\mf{c}+\mf{d}+2)}
& \text{$n$ odd}
\end{cases},
\\
C_{n}&=
\begin{cases}
-\frac{n(n+2\mf{c}+2\mf{d}+1)}{2(n+\mf{a}+\mf{b}+\mf{c}+\mf{d}+1)} & \text{$n$ even}
\\
-\frac{(n+2\mf{b}+2\mf{c}+1)(n+2\mf{b}+2\mf{d}+1)}{2(n+\mf{a}+\mf{b}+\mf{c}+\mf{d}+1)} & \text{$n$ odd}
\end{cases}.
\end{align*}
The polynomials $\mb{B}_{n}(x)$ are eigenfunctions of the most general, self-adjoint, first-order Dunkl shift operator preserving the space of polynomials of a given degree \cite{2012_Tsujimoto&Vinet&Zhedanov_AdvMath_229_2123}. Specifically, they satisfy the eigenvalue equation
\begin{align}
\label{BI-Eigen}
\mathcal{L}\mb{B}_{n}(x)=\lambda_{n}\mb{B}_{n}(x),\qquad \lambda_n=(-1)^{n}(n+\mf{a}+\mf{b}+\mf{c}+\mf{d}+3/2),
\end{align}
where $\mathcal{L}$ is the difference operator given by
\begin{multline}
\label{BI-Operator}
\mathcal{L}=\left(\frac{(x+2\mf{c}+1)(x+2\mf{d}+1)}{2x+1}\right)(T^{+}_{x}R_{x}-\mathbb{1})
\\
-\left(\frac{(x-2\mf{a}-1)(x-2\mf{b}-1)}{2x-1}\right)(T^{-}_{x}R_{x}-\mathbb{1})
+(\mf{a}+\mf{b}+\mf{c}+\mf{d}+3/2)\mathbb{1},
\end{multline}
where $R_{x}f(x)=f(-x)$ is the reflection operator and where $T^{\pm}_{x}f(x)=f(x\pm 1)$ are the discrete shift operators.  The eigenvalue equation \eqref{BI-Eigen}, together with the monicity condition, completely determine the polynomials $\mb{B}_{n}(x)$.  The bispectrality property of $\mb{B}_{n}(x)$, expressed through the relations \eqref{BI-3Term} and \eqref{BI-Eigen}, is encoded in an algebraic structure that has been called the Bannai--Ito algebra \cite{2012_Tsujimoto&Vinet&Zhedanov_AdvMath_229_2123}; this associative algebra is obtained as follows. Let $K_1$, $K_2$ be the generators defined by
\begin{align}
\label{rr-1}
K_1=\mathcal{L},\quad K_2=X,
\end{align}
where $X$ is the ``multiplication by $x$'' operator. Writing the anticommutator as $\{A,B\}=AB+BA$, one has \cite{2012_Tsujimoto&Vinet&Zhedanov_AdvMath_229_2123}
\begin{align}
\label{BI-Algebra}
\{K_1,K_2\}=K_3+\omega_3,\quad \{K_2,K_3\}=K_1+\omega_1,\quad \{K_3,K_1\}=K_2+\omega_2,
\end{align}
where the first relation of \eqref{BI-Algebra} is taken to define $K_3$. The structure constants $\omega_i$ have the expressions
\begin{align*}
\omega_1&=4(\mf{ab}+\mf{cd})+(\mf{a}+\mf{b}+\mf{c}+\mf{d})+1/2,
\\
\omega_2&=2(\mf{a}^2+\mf{b}^2-\mf{c}^2-\mf{d}^2)+(\mf{a}+\mf{b}-\mf{c}-\mf{d}),
\\
\omega_3&=4(\mf{ab}-\mf{cd})+(\mf{a}+\mf{b}-\mf{c}-\mf{d}).
\end{align*}
The relations  \eqref{BI-Algebra} define the abstract Bannai--Ito algebra. Note that the arbitrary structure constants can be viewed as being associated to a central element. This algebra has a Casimir operator
\begin{align}
\label{BI-Casimir}
Q=K_1^2+K_2^2+K_3^2,
\end{align}
which commutes with all generators $K_i$, $i=1,2,3$. In the realization \eqref{rr-1}, the Casimir operator takes the value
\begin{align}
\label{Cas-Value}
Q=2(\mf{a}^2+\mf{b}^2+\mf{c}^2+\mf{d}^2)+(\mf{a}+\mf{b}+\mf{c}+\mf{d})+1/4.
\end{align}
When the parameters $\mf{a},\mf{b},\mf{c},\mf{d}$ are real, the Bannai--Ito polynomials can only obey a finite orthogonality relation with respect to a discrete, positive measure; see \cite{2013_Genest&Vinet&Zhedanov_SIGMA_9_18} for details. This is related to the fact that for real values of the parameters, the self-adjoint representations of the Bannai--Ito algebra with spectrum \eqref{BI-Eigen}, i.e. those that satisfy the $*$-relations $K_i^{\dagger}=K_i$ for $i=1,2,3$, can only be finite-dimensional. Let us note here that the Bannai--Ito algebra has recently appeared as symmetry algebra for the Laplace--Dunkl \cite{2015_Genest&Vinet&Zhedanov_CommMathPhys_336_243} and Dirac--Dunkl \cite{2015_DeBie&Genest&Vinet_ArXiv_1501.03108} equations on the 2-sphere as well as in the Racah problem of  $\mf{osp}(1|2)$ \cite{2014_Genest&Vinet&Zhedanov_ProcAmMathSoc_142_1545}. Let us also mention that the Bannai--Ito polynomials sit at the top of the emerging ``$-1$'' scheme of orthogonal polynomials, which extends the Askey tableau; see \cite{2012_Tsujimoto&Vinet&Zhedanov_AdvMath_229_2123,2013_Genest&Vinet&Zhedanov_SIGMA_9_18, 2013_Tsujimoto&Vinet&Zhedanov_ProcAmMathSoc_141_959, 2012_Tsujimoto&Vinet&Zhedanov_TransAmerMathSoc_364_5491}.

The Bannai--Ito algebra can be thought of as a $q=-1$ analog of the Askey--Wilson algebra \cite{1991_Zhedanov_TheorMathPhys_89_1146}, which describes the deeper symmetries of the Askey--Wilson polynomials. In two papers \cite{2007_Koornwinder_SIGMA_3_63, 2008_Koornwinder_SIGMA_4_52}, Koornwinder established the relationship between the Askey--Wilson algebra and the double affine Hecke algebra associated with the rank-one root system of type $(C_1^{\vee}, C_1)$. These results relate to the large body of work on multivariate orthogonal polynomials and generalized Calogero-Moser systems associated with Cherednik algebras and root systems; see for example \cite{1995_VanDiejen_CompMath_95_183, 1992_Koornwinder_ContempMath_138, MacDonald-1998, 2000_VanDiejen&Vinet}. In the present paper, we shall put the Bannai--Ito algebra and polynomials in a similar framework by exhibiting their relationship with the degenerate double affine Hecke algebra associated to the rank one root system of type $(C_1^{\vee}, C_1)$ and with the non-symmetric Wilson polynomials introduced by Groenevelt in \cite{2007_Groenevelt_TransGroups_12_77, 2009_Groenevelt_SelMath_15_377}.
\subsection{The one-variable non-symmetric Wilson polynomials and a rank-one, degenerate double affine Hecke algebra}
Let $t_0, t_1$ and $u_0,u_1$ be complex parameters and let $T_0, T_1$ be defined by
\begin{align}
\label{rea-1}
\begin{aligned}
T_0&=\frac{(t_0+u_0-z+1/2)(t_0-u_0-z+1/2)}{1-2z}(T^{-}_{z}R_{z}-\mathbb{1})+t_0 \mathbb{1},
\\
T_1&=\frac{(t_1+u_1+z)(t_1-u_1+z)}{2z}(R_{z}-\mathbb{1})+t_1 \mathbb{1}.
\end{aligned}
\end{align}
It is easy to check that $T_i^2=t_i^2$ for $i=0,1$; hence $T_0$ and $T_1$ are involutions. The non-symmetric Wilson polynomials $p_{n}(z)$ of degree $n$ in $z$, which depend on four parameters $t_0$, $t_1$, $u_0$, $u_1$, are defined as the unique (monic) polynomials satisfying the eigenvalue equation \cite{2007_Groenevelt_TransGroups_12_77}
\begin{align}
\label{Wilson-Eigen}
(T_0+T_1)\,p_{n}(z)=\gamma_{n}p_{n}(z),
\end{align}
where the eigenvalues $\gamma_n$ are given by
\begin{align}
\label{Wilson-Eigen-2}
\gamma_{n}=
\begin{cases}
t_0+t_1+m & n=2m
\\
-(t_0+t_1+m) & n=2m-1
\end{cases}.
\end{align}
The polynomials $p_{n}(z)$ satisfy the complex orthogonality relation \cite{2007_Groenevelt_TransGroups_12_77}
\begin{align}
\label{Ortho}
\frac{1}{2\pi i}\int_{\mathcal{C}}p_{n}(z)p_{m}(z) \,\Delta(z)\,\mathrm{d}z=h_{n}\delta_{nm},
\end{align}
with respect to the weight function
\begin{multline*}
\Delta(z)=
\frac{\Gamma(t_1+u_1+z)\Gamma(t_1+u_1+1-z)\Gamma(t_1-u_1+z)\Gamma(t_1-u_1+1-z)}{\Gamma(2z)\Gamma(1-2z)}
\\
\times \Gamma(t_0+u_0+1/2+z)\Gamma(t_0+u_0+1/2-z)\Gamma(t_0-u_0+1/2+z)\Gamma(t_0-u_0+1/2-z),
\end{multline*}
where $h_{n}\neq 0$ and where $\mathcal{C}$ is the usual contour that runs along the imaginary axis and that is indented, if necessary, to separate the increasing sequence of poles from the decreasing sequence of poles in the weight function; see for example \cite{2001_Andrews&Askey&Roy}. Sufficient conditions on the parameters $t_i$, $u_i$ for this contour to exist are easily found. The algebraic structure associated with the non-symmetric Wilson polynomials is as follows. Introduce the operators $U_0$, $U_1$ defined as \cite{2007_Groenevelt_TransGroups_12_77}
\begin{align}
\label{rea-2}
\begin{aligned}
U_0=-T_0+Z-1/2,\qquad 
U_1=-T_1-Z,
\end{aligned}
\end{align}
where $Z$ is the ``multiplication by $z$'' operator. The degenerate double affine Hecke algebra associated to the non-symmetric Wilson polynomials has generators $T_0$, $T_1$, $U_0$, $U_1$ and relations
\begin{gather}
\label{Algebra}
\begin{gathered}
T_0^2=t_0^2,\quad T_1^2=t_1^2,\quad U_0^2=u_0^2,\quad U_1^2=u_1^2,
\\
T_0+T_1+U_0+U_1=-\frac{1}{2}.
\end{gathered}
\end{gather}
These relations are easily verified in the realization \eqref{rea-1}, \eqref{rea-2}. This algebra can be viewed as a $q=1$ analog of the double affine Hecke algebra of type $(C_1^{\vee}, C_1)$ \cite{2000_Stokman_IntMathResNot_1005, 1999_Sahi_AnnMath_150_267}. The algebra \eqref{Algebra} was also considered in \cite{2007_Etingof&Oblomkov&Rains_AdvMath_212_749}.  In the following, we shall use an elementary generalization \eqref{Algebra} for which $t_i$, $u_i$ for $i=0,1$ are considered as central elements rather than complex parameters; following \cite{2013_Terwilliger_Sigma_9_47}, this could be called the ``universal'' degenerate double affine Hecke algebra of type $(C_1^{\vee}, C_1)$.
\subsection{Outline}
The outline for the rest of the paper is straightforward. In Section 2, the non-symmetric Wilson polynomials are shown to coincide with the Bannai--Ito polynomials by comparison of their respective eigenvalue equations. The isomorphism between the (universal) degenerate double affine Hecke algebra and the Bannai--Ito algebra is given explicitly on the generators. In Section 3, we exhibit a positive-definite continuous orthogonality measure for the BI polynomials. We also display a non-compact form of the BI algebra and specify a four-parameter family of irreducible self-adjoint representations. A conclusion follows.
\section{Bannai--Ito vs. non-symmetric Wilson polynomials}
In this section, the equivalence between the Bannai--Ito polynomials and the non-symmetric Wilson polynomials is demonstrated and the isomorphism between the Bannai--Ito algebra and the (universal) degenerate double affine Hecke algebra is established.
\begin{prop}
Let $p_{n}(z;t_0, t_1, u_0,u_1)$ be the monic non-symmetric Wilson polynomials defined by the eigenvalue equation \eqref{Wilson-Eigen} and let $\mb{B}_{n}(x;\mf{a},\mf{b},\mf{c},\mf{d})$ be the monic Bannai--Ito polynomials defined by \eqref{BI-Eigen}. One has
\begin{align}
(-2)^{n}p_{n}\left(-\frac{x}{2}+\frac{1}{4}; \frac{\mf{c}+\mf{d}}{2}+\frac{1}{4}, \frac{\mf{a}+\mf{b}}{2}+\frac{1}{4}, \frac{\mf{c}-\mf{d}}{2},\frac{\mf{a}-\mf{b}}{2}\right)=\mb{B}_{n}(x;\mf{a},\mf{b},\mf{c},\mf{d}).
\end{align}
Hence the non-symmetric Wilson polynomials coincide with the Bannai--Ito polynomials, up to an affine transformation.
\end{prop}
\begin{proof}
As $p_{n}(z;t_0,t_1,u_0, u_1)$ and $\mb{B}_{n}(x;\mf{a},\mf{b},\mf{c},\mf{d})$ are determined uniquely by their eigenvalue equations \eqref{Wilson-Eigen} and \eqref{BI-Eigen} and the monicity condition, it suffices to show that under the affine transformation and reparametrization
\begin{align}
z\rightarrow -\frac{x}{2}+\frac{1}{4},\quad (t_0, t_1,u_0, u_1)\rightarrow \left(\frac{\mf{c}+\mf{d}}{2}+\frac{1}{4}, \frac{\mf{a}+\mf{b}}{2}+\frac{1}{4}, \frac{\mf{c}-\mf{d}}{2},\frac{\mf{a}-\mf{b}}{2}\right),
\end{align}
the eigenvalue equation \eqref{Wilson-Eigen} gives \eqref{BI-Eigen}. Under the above affine transformation, it is easily checked that one has
\begin{align*}
T_{z}^{-}R_{z}\rightarrow T_{x}^{+}R_{x},\qquad R_{z}\rightarrow T_{x}^{-}R_{x},
\end{align*}
and hence that \eqref{Wilson-Eigen} becomes
\begin{multline*}
\Bigg[\frac{(x+2\mf{c}+1)(x+2\mf{d}+1)}{2x+1}(T_{x}^{+}R_{x}-\mathbb{1})+(\mf{a}+\mf{b}+\mf{c}+\mf{d}+1)\mathbb{1}
\\
+\frac{(x-2\mf{a}-1)(x-2\mf{b}-1)}{2x+1}(T_{x}^{-}R_{x}-\mathbb{1})\Bigg]\,p_{n}\left(-\frac{x}{2}+\frac{1}{4}\right)=2\,\gamma_{n}\,p_{n}\left(-\frac{x}{2}+\frac{1}{4}\right).
\end{multline*}
Upon adding $(1/2)\mathbb{1}$ on both sides of the above equation and using the expression \eqref{Wilson-Eigen-2} for the eigenvalues, one finds \eqref{BI-Eigen}.
\end{proof}
\begin{prop}
The Bannai--Ito algebra with defining relations \eqref{BI-Algebra} is isomorphic to the (universal) degenerate double affine Hecke algebra \eqref{Algebra}.
\end{prop}
\begin{proof}
Let $K_1$, $K_2$, $K_3$ satisfy the Bannai--Ito algebra \eqref{BI-Algebra} with arbitrary structure constants or central elements $\omega_1$, $\omega_2$, $\omega_3$ and consider the following combinations
\begin{alignat}{2}
\label{Alg-Map}
\begin{aligned}
\widetilde{T}_0&=\frac{1}{4}\left(K_1-K_2-K_3-\frac{1}{2}\right),\quad& \widetilde{T}_1&=\frac{1}{4}\left(K_1+K_2+K_3-\frac{1}{2}\right),
\\
\widetilde{U}_0&=\frac{1}{4}\left(-K_1-K_2+K_3-\frac{1}{2}\right),\quad& \widetilde{U}_1&=\frac{1}{4}\left(-K_1+K_2-K_3-\frac{1}{2}\right).
\end{aligned}
\end{alignat}
It is easily verified that each of $\widetilde{T}_0^2$, $\widetilde{T}_1^2$, $\widetilde{U}_0^2$ $\widetilde{U}_1^2$ is central; i.e that they commute with all generators $\widetilde{T}_0$, $\widetilde{T}_1$, $\widetilde{U}_0$, $\widetilde{U}_1$. Indeed, one has
\begin{alignat*}{2}
\widetilde{T}_0^2&=\frac{1}{16}\left(Q+\omega_1-\omega_2-\omega_3+\frac{1}{4}\right)=\widetilde{t}_0,\;&
\widetilde{T}_1^2&=\frac{1}{16}\left(Q+\omega_1+\omega_2+\omega_3+\frac{1}{4}\right)=\widetilde{t}_1,
\\
\widetilde{U}_0^2&=\frac{1}{16}\left(Q-\omega_1-\omega_2+\omega_3+\frac{1}{4}\right)=\widetilde{u}_0,\;&
\widetilde{U}_1^2&=\frac{1}{16}\left(Q-\omega_1+\omega_2-\omega_3+\frac{1}{4}\right)=\widetilde{u}_1.
\end{alignat*}
Moreover, it is obvious that
\begin{align*}
\widetilde{T}_0+\widetilde{T}_1+\widetilde{U}_0+\widetilde{U}_1=-1/2.
\end{align*}
Hence \eqref{Alg-Map} provides an algebra map from the Bannai--Ito algebra to the universal degenerate double affine Hecke algebra \eqref{Algebra}. Let us now exhibit the inverse mapping.  Let $T_i$ and $U_i$ for $i=0,1$ satisfy \eqref{Algebra} and consider the linear combinations
\begin{align}
\label{Dompe}
\begin{aligned}
A_1=2T_0+2T_1+1/2,\; A_2=-2T_0-2U_0-1/2,\;
A_3=2T_1+2U_0+1/2.
\end{aligned}
\end{align}
Then by a direct calculation, one finds
\begin{align*}
\{A_1,A_2\}&=A_3+4(t_1^2-t_0^2+u_0^2-u_1^2),
\\
\{A_2,A_3\}&=A_1+4(t_1^2+t_0^2-u_0^2-u_1^2),
\\
\{A_3,A_1\}&=A_2+4(t_1^2-t_0^2-u_0^2+u_1^2).
\end{align*}
Hence the combinations \eqref{Dompe} realize the Bannai--Ito relations \eqref{BI-Algebra}. In the realization \eqref{Dompe}, the Casimir operator of the Bannai--Ito algebra has the expression
\begin{align*}
A_1^2+A_2^2+A_3^2=4(t_0^2+t_1^2+u_0^2+u_1^2)-1/4.
\end{align*}
Hence the Bannai--Ito algebra \eqref{BI-Algebra} and the degenerate double affine Hecke algebra \eqref{Algebra} are isomorphic.
\end{proof}
\section{``Continuous'' Bannai--Ito polynomials and a non-compact form of the Bannai--Ito algebra}
In this section, we exhibit a positive-definite orthogonality measure on the real line for the Bannai--Ito polynomials. We relate the existence of this measure to a non-compact form of the BI algebra and present a four-parameter family of irreducible representations stemming from the BI recurrence relation.

Consider the modified Bannai--Ito polynomials, denoted by $\mb{Q}_{n}(x;\mf{a},\mf{b},\mf{c},\mf{d})$, that are obtained by taking
\begin{align}
\label{Modified-BI}
\mb{Q}_{n}(x;\mf{a},\mf{b},\mf{c},\mf{d})=(-i)^{n}\mb{B}_{n}(i z;\mf{a},\mf{b},\mf{c},\mf{d}).
\end{align}
It follows from \eqref{BI-3Term} that the polynomials $\mb{Q}_{n}(x)$ satisfy a three-term recurrence relation of the form
\begin{align}
\label{Q-Recu}
x\mb{Q}_{n}(x)=\mb{Q}_{n+1}(x)+c_{n}\mb{Q}_{n}(x)+u_{n}\mb{Q}_{n}(x).
\end{align}
where the recurrence coefficients are given by
\begin{align*}
c_n=-i(2\mf{a}+1-A_{n}-C_{n}),\qquad u_{n}=-A_{n-1}C_{n}.
\end{align*}
It is verified that the polynomials $\mb{Q}_{n}(x;\mf{a},\mf{b},\mf{c},\mf{d})$ enjoy the symmetries 
\begin{gather*}
\begin{gathered}
\mb{Q}_{n}(x;\mf{a},\mf{b},\mf{c},\mf{d})=\mb{Q}_{n}(x;\mf{b},\mf{a},\mf{c},\mf{d}),\quad 
\mb{Q}_{n}(x;\mf{a},\mf{b},\mf{c},\mf{d})=\mb{Q}_{n}(x;\mf{a},\mf{b},\mf{d},\mf{c}),
\\
\mb{Q}_{n}(x;\mf{a},\mf{b},\mf{c},\mf{d})=(-1)^{n}\mb{Q}_{n}(-x;\mf{c},\mf{d},\mf{a},\mf{b}).
\end{gathered}
\end{gather*}
Assume now that the parameters $\mf{a}$, $\mf{b}$, $\mf{c}$ and $\mf{d}$ are such that
\begin{align}
\label{CC}
\bar{\mf{a}}=\mf{c}\,\text{or}\,\mf{d},\qquad \bar{\mf{b}}=\mf{d}\,\text{or}\,\mf{c},
\end{align}
where $\bar{x}$ stands for complex conjugation. This condition can be implemented, for example, by taking
\begin{align}
\label{para-1}
\mf{a}=\alpha +i \beta,\quad \mf{b}=\gamma+i\delta,\quad \mf{c}=\alpha-i\beta,\quad \mf{d}=\gamma-i\delta,
\end{align}
where $\alpha$, $\beta$, $\gamma$, $\delta$ are real parameters. Under the parametrization \eqref{para-1}, it is directly verified that the recurrence coefficients $c_{n}$ are of the form
\begin{align}
\label{Cn}
c_{n}=
\begin{cases}
2\beta-\frac{(n+4\alpha+2)(\beta-\delta)}{(n+2\alpha+2\gamma+2)}-\frac{n(\beta+\delta)}{(n+2\alpha+2\gamma+1)} & $\text{$n$} even$
\\
2\beta-\frac{(n+4\alpha+4\gamma+3)(\beta+\delta)}{n+2\alpha+2\gamma+2}-\frac{(n+4\gamma+1)(\beta-\delta)}{n+2\alpha+2\gamma+1} & $\text{$n$} odd$
\end{cases},
\end{align}
and hence real for all $n=0,1,\ldots$. The coefficients $u_{n}$ are of the form
\begin{align}
\label{Un}
u_{n}=
\begin{cases}
\frac{n(n+4\alpha+4\gamma+2)\,\rVert n+2[\alpha+\gamma+i(\beta+\delta)]+1 \rVert^2}{4(n+2\alpha+2\gamma+1)^2} & \text{$n$ even}
\\
\frac{(n+4\alpha+1)(n+4\gamma+1)\,\rVert n+2[\alpha+\gamma+i(\beta-\delta)] +1\rVert^2}{4(n+2\alpha+2\gamma+1)^2} & \text{$n$ odd}
\end{cases}.
\end{align}
It is seen that if $\alpha$, $\beta$, $\gamma$, $\delta$ are positive, then $u_{n}>0$ for all $n=0,1,\ldots$. It hence follows from general theory that the polynomials $\mb{Q}_{n}(x)$ form an infinite family of orthogonal polynomials with respect to a positive-definite measure \cite{Ismail-2009}.
\begin{prop}
Let $\mf{a}$ and $\mf{b}$ be complex numbers with positive real and imaginary parts and let $\mf{c}=\bar{\mf{a}}\text{ or } \bar{\mf{b}}$ and $\mf{d}=\bar{\mf{b}} \text{ or } \bar{\mf{a}}$. The Bannai--Ito polynomials $\mb{B}_{n}(x;\mf{a},\mf{b},\mf{c},\mf{d})$ satisfy the orthogonality relation
\begin{align}
\frac{1}{4\pi}\int_{-\infty}^{\infty}W(z)\,\mb{B}_{n}(z)\,\mb{B}_{m}(z)\;\mathrm{d}z=h_{0}\,\delta_{nm}\,\prod_{k=1}^{n}u_{k},
\end{align}
where the positive weight function $W(z)$ is given by
\small
\begin{align*}
W(z)=\Bigg\rvert \frac{\Gamma(\mf{a}+iz/2+1)\Gamma(\mf{b}+iz/2+1)\Gamma(\mf{c}+iz/2+1/2)\Gamma(\mf{d}+iz/2+1/2)}{\Gamma(1/2+iz)} \Bigg\rvert^2,
\end{align*}
\normalsize
and where $h_{0}$ reads
\small
\begin{align*}
h_{0}=\frac{\Gamma(\mf{a}+\mf{b}+3/2)\Gamma(\mf{a}+\mf{c}+1)\Gamma(\mf{b}+\mf{c}+1)\Gamma(\mf{a}+\mf{d}+1)\Gamma(\mf{b}+\mf{d}+1)\Gamma(\mf{c}+\mf{d}+3/2)}{\Gamma(\mf{a}+\mf{b}+\mf{c}+\mf{d}+2)}.
\end{align*}
\normalsize
\end{prop}
\begin{proof}
The orthogonality property follow directly from \eqref{Ortho} while the normalization coefficients follow from general theory and from comparison with the Wilson integral \cite{2010_Koekoek_&Lesky&Swarttouw}; see also \cite{2007_Groenevelt_TransGroups_12_77}.
\end{proof}
From \eqref{BI-Eigen}, it follows that the modified Bannai--Ito polynomials $\mb{Q}_{n}(x;\mf{a},\mf{b},\mf{c},\mf{d})$ defined by \eqref{Modified-BI} satisfy the eigenvalue equation
\begin{align}
\label{Q-Eigen}
\mathcal{M}\mb{Q}_{n}(x)=\lambda_{n}\mb{Q}_{n}(x),\qquad \lambda_n=(-1)^{n}(n+\mf{a}+\mf{b}+\mf{c}+\mf{d}+3/2),
\end{align}
where $\mathcal{M}$ is given by
\begin{multline}
\label{OP-M}
\mathcal{M}=\left(\frac{(2\mf{a}+1-ix)(2\mf{b}+1-ix)}{1-2ix}\right)(S_{x}^{+}R_{x}-\mathbb{1})
\\
+\left(\frac{(2\mf{c}+1+ix)(2\mf{d}+1+ix)}{1+2ix}\right)(S_{x}^{-}R_{x}-\mathbb{1})+(\mf{a}+\mf{b}+\mf{c}+\mf{d}+3/2)\mathbb{1},
\end{multline}
where $S_{x}^{\pm}f(x)=f(x\pm i)$ are the imaginary shift operators. It is manifest that under the conditions \eqref{CC}, the eigenvalues $\lambda_n$ in \eqref{Q-Eigen} are real and the operator $\mathcal{M}$ is self-adjoint. The operators $\mathcal{M}$ and $X$ generate a non-compact form of the Bannai--Ito algebra. Indeed, upon defining
\begin{align}
\label{rea-3}
A_1=\mathcal{M},\quad A_2=X,
\end{align}
a direct calculation shows that one has
\begin{align}
\label{NC-BI}
\{A_1,A_2\}=A_3+\alpha_3,\quad \{A_2,A_3\}=-A_1+\alpha_1,\quad \{A_3,A_1\}=A_2+\alpha_2,
\end{align}
where the first relation of \eqref{NC-BI} is taken to define $A_3$ and where the structure constants read
\begin{align}
\label{Struc}
\begin{aligned}
\alpha_1&=-[4(\mf{ab}+\mf{cd})+\mf{a}+\mf{b}+\mf{c}+\mf{d}+1/2],
\\
\alpha_2&=-i[2(\mf{a}^2+\mf{b}^2-\mf{c}^2-\mf{d}^2)+\mf{a}+\mf{b}-\mf{c}-\mf{d}],
\\
\alpha_3&=-i[4(\mf{ab}-\mf{cd})+\mf{a}+\mf{b}-\mf{c}-\mf{d}].
\end{aligned}
\end{align}
When $\mf{c}=\bar{\mf{a}}\text{ or } \bar{\mf{b}}$ and $\mf{d}=\bar{\mf{b}}\text{ or } \bar{\mf{a}}$, the structure constants \eqref{Struc} are real. The algebra \eqref{NC-BI} differs from the standard Bannai--Ito algebra \eqref{BI-Algebra} by a change of sign in the anticommutation relations and can thus be viewed  as a non-compact form of the latter. The difference between  \eqref{BI-Algebra} and \eqref{NC-BI}  is similar in spirit to the difference that exists between $\mathfrak{su}(2)$ and $\mathfrak{su}(1,1)$. The Casimir operator $Z$ for the algebra \eqref{NC-BI} is naturally given by
\begin{align*}
Z=A_1^2-A_2^2-A_3^2.
\end{align*}
In the realization \eqref{rea-3}, it takes the same value as in \eqref{Cas-Value}. 

In light of the eigenvalue equation \eqref{Q-Eigen} and the recurrence relation \eqref{Q-Recu}, one can straightforwardly introduce a four-parameter family of infinite-dimensional, self-adjoint representations of the algebra \eqref{NC-BI}. This is done in the following proposition.
\begin{prop}
Let $\alpha,\beta,\gamma,\delta$ be positive real numbers and consider the infinite-dimensional vector space  $V$ spanned by the basis vectors $e_{n}$, $n=0,1,\ldots$, endowed with the actions
\begin{align}
\label{Actions}
\begin{aligned}
A_1\,e_{n}&=(-1)^{n}(n+2\alpha+2\gamma+3/2)\,e_{n},
\\
A_2\,e_{n}&=\sqrt{u_{n+1}}\,e_{n+1}+c_{n}\,e_{n}+\sqrt{u_{n}}\,e_{n-1},
\end{aligned}
\end{align}
where $u_{n}$ and $c_{n}$ are given by \eqref{Cn} and \eqref{Un}, respectively. Then $V$ supports an infinite-dimensional, irreducible, self-adjoint representation of the non-compact Bannai--Ito algebra \eqref{NC-BI}.
\end{prop}
\begin{proof}
One can verify by direct calculation that with the actions \eqref{Actions}, the generators $A_1$, $A_2$ satisfy the relations \eqref{NC-BI}. The fact that the representation is self-adjoint follows from the observation that for positive values of $\alpha,\beta,\gamma,\delta$, the coefficients $c_{n}$ are real and the coefficients $u_{n}$ are positive for all $n=0,1,\ldots$. The irreducibility is a consequence of the fact that under these conditions, $u_{n}\neq 0$ for all non-negative integers $n$.
\end{proof}
\section{Conclusion}
In this paper, we have established the relationship between the Bannai--Ito polynomials and the non-symmetric Wilson polynomials as well as between the Bannai--Ito algebra and a rank-one degenerate double affine Hecke algebra associated with the root system $(C_1^{\vee}, C_1)$. We have also exhibited a positive-definite, continuous measure for the BI polynomials, displayed the associated non-compact form of the BI algebra and presented a family of irreducible representations.

The results presented suggest several questions of interest. First, given the applications of the one-variable Bannai--Ito polynomials to exactly solvable models (e.g. \cite{2015_Genest&Vinet&Zhedanov_CommMathPhys_336_243,2015_DeBie&Genest&Vinet_ArXiv_1501.03108}), it would be of interest to study the multivariate non-symme\-tric Wilson polynomials introduced in \cite{2009_Groenevelt_SelMath_15_377} from that perspective; the results of such an investigation should be compared for example with those found in \cite{1995_VanDiejen_JPhysA_28_369} and \cite{2013_VanDiejen&Emsiz_JFunAnal_265_1981}. Second, in view of the fact that the Bannai--Ito polynomials $\mb{B}_{n}(x)$ obeying a discrete and finite orthogonality relation arise as Racah coefficients for positive-discrete series representations of $\mf{osp}(1|2)$ \cite{2014_Genest&Vinet&Zhedanov_ProcAmMathSoc_142_1545}, it would be natural to look for a similar interpretation for the modified Bannai--Ito polynomials $\mb{Q}_{n}(x)$, which satisfy a continuous orthogonality relation. This would possibly involve different types of representations, other than those of the positive-discrete series; see \cite{2006_Groenevelt_ActaApplMath_91_133}
\section*{Acknowledgments}\noindent
The authors would like to thank S. Tsujimoto and W. Groenevelt for discussions. VXG is supported by the Natural Science and Engineering Research Council of Canada (NSERC). The research of LV is supported in part by NSERC. AZ would like to thank the Centre de recherches math\'ematiques (CRM) for its hospitality.

\end{document}